\documentclass[12pt]{article}

\usepackage[colorlinks=true, pdfstartview=FitV, linkcolor=blue, citecolor=blue, urlcolor=blue]{hyperref}

\usepackage{amssymb,amsmath, amscd, amsthm}
\usepackage{times, verbatim}
\usepackage{graphicx}
\usepackage{amssymb,latexsym,comment,url}

\newcommand\Z{{\mathbb Z}}
\newcommand\Q{{\mathbb Q}}
\newcommand\R{{\mathbb R}}
\renewcommand\P{{\mathbb P}}
\newcommand{\F}{{\mathbb F}}
\newcommand{\Fp}{{\mathbb F}_p}
\newcommand{\GL}{\operatorname{GL}}
\newcommand{\Fbar}{\overline{F}}
\newcommand{\fbar}{\overline{f}}
\newcommand{\fpbar}{\overline{f_1}}
\renewcommand{\hbar}{\overline{h}}
\newcommand{\Qp}{{\mathbb Q}_p}
\newcommand{\Zp}{{\mathbb Z}_p}
\newcommand{\vol}{\operatorname{vol}}
\newcommand{\oldalpha}{\beta}
\newcommand{\oldbeta}{\alpha}
\newcommand\mC{{C}} 
\newcommand\Cbar{\overline{C}}

                        {\hspace*{\fill}\nobreak$\Box$ \vspace{2 ex}}

\DeclareFontFamily{OT1}{rsfs}{}
\DeclareFontShape{OT1}{rsfs}{n}{it}{<-> rsfs10}{}
\DeclareMathAlphabet{\mathscr}{OT1}{rsfs}{n}{it}

\newtheorem{theorem}{Theorem}
\newtheorem{proposition}[theorem]{Proposition}
\newtheorem{corollary}[theorem]{Corollary}
\newtheorem{lemma}[theorem]{Lemma}

\title{The proportion of genus one curves over $\Q$ defined by
  a binary quartic that everywhere locally have a point}

\author{Manjul Bhargava, John Cremona, and Tom Fisher}

\date{27th July 2020}

\begin{document}

\maketitle

\begin{abstract}
  We consider the proportion of genus one curves over
  $\Q$ of the form $z^2=f(x,y)$ where $f(x,y)\in\Z[x,y]$ is a binary
  quartic form (or more generally of the form $z^2+h(x,y)z=f(x,y)$
  where also $h(x,y)\in\Z[x,y]$ is a binary quadratic form) that have
  points everywhere locally.
  We show that the proportion of these curves
  that are locally soluble, computed as a product of local
  densities, is approximately 75.96\%.  We prove that the local density
  at a prime $p$ is given by a fixed degree-$9$ rational function of
  $p$ for all odd $p$ (and for the generalized equation, the same
  rational function gives the local density at every prime).  An
  additional analysis is carried out to estimate rigorously the local
  density at the real place.
\end{abstract}

\section{Introduction}
\label{sec:Intro}

In this paper we show that most genus one curves over $\Q$ of the form
$z^2=f(x,y)$, where $f\in\Z[x,y]$ is a binary quartic form, have
a point everywhere locally.

Consider the family of equations
\begin{equation}\label{hypereq}
z^2 = f(x,y) = ax^4+bx^3y+cx^2y^2+dxy^3+ey^4,
\end{equation}
where $a,b,c,d,e\in\Z$.  Provided that $f$ is squarefree, such an
equation defines a genus one curve over~$\Q$.  We define the \emph{height}
of the equation~\eqref{hypereq} by
\begin{equation}\label{def:H}
H(f):=\max\{|a|,|b|,|c|,|d|,|e|\}.
\end{equation}
Let $\rho$ denote the density of equations (\ref{hypereq}), when
ordered by height, that are everywhere locally soluble;
that is,
\begin{equation}\label{def:rho}
\rho = \lim_{X\to\infty} \frac{\#\{f \mid H(f) \le
  X,\ \text{(\ref{hypereq}) everywhere locally soluble}\}}
                    {\#\{f \mid H(f) \le X\}}.
\end{equation}
Since 100\% of binary quartics over~$\Z$, when ordered by height, are
squarefree, $\rho$ also represents the density of genus one curves~$C$ of the
form $z^2=f(x,y)$ that have a point everywhere locally.
In this paper we describe how to compute $\rho$, and in particular
show that $\rho \approx 75.96\%$.

We define local densities for solubility as follows.  For $p$ a prime,
we identify the space of equations~\eqref{hypereq} having coefficients
in $\Zp$ with the space $\Zp^{5}$ equipped with its natural additive
Haar measure.  For measurable subsets~$S$ of~$\Zp^5$, or $\Zp^n$ for
any~$n\ge1$, we will also refer to the measure of~$S$ as its \emph{density},
or as the \emph{probability} that a random element of~$\Zp^n$ lies in~$S$.
We then write $\rho(p)$ for the density of equations~\eqref{hypereq}
over $\Zp$ that have a solution over~$\Qp$.
We also write $\rho(\infty)$ for the probability that an equation of
the form~\eqref{hypereq}, with real coefficients independently and
uniformly distributed in $[-1,1]$, has a real solution.

It is then a theorem of Poonen and Stoll~\cite{PS2} (which in turn
relies on the sieve of Ekedahl~\cite{Ek}) that the density~$\rho$
exists and is given by
\begin{equation}\label{prodformula}
\rho=\rho(\infty)\prod_{p \text{ prime }}\rho(p).
\end{equation}
We prove the following theorem.

\begin{theorem}\label{maingenusonelocal}
  Let $\rho(p)$ denote the density of binary quartic forms $f$ over
  $\Zp$ such that $z^2=f(x,y)$ is soluble over $\Qp$.  Then
  \[\rho(p) = \left\{\begin{array}{cl}
  \displaystyle\frac{23087}{24528}& \mbox{if $p=2;$}\\[.125in]
  R(p) & \mbox{if $p\ge3$,}\end{array}\right.\]
  where $R$ is the rational function
\begin{equation}
\label{R1-formula}
R(t)
 = 1 - \frac{4 t^{7} + 4 t^{6} + 2 t^{5} + t^{4} + 3 t^{3} + 2 t^{2} + 3 t + 3}{8(t + 1) (t^{2} + t + 1)  (t^{6} + t^{3} + 1)}.
\end{equation}
\end{theorem}

We also carry out a rigorous numerical integration to prove the
following.
\begin{proposition}  \label{prop:R-density-bounds}
Let $\rho(\infty)$ denote the probability that an
equation~$z^2=f(x,y)$, where $f$ is a binary quartic form with real
coefficients independently and uniformly distributed in $[-1,1]$, has
a real solution. Then
\[
  0.873954 \le \rho(\infty) \le 0.874124.
\]
\end{proposition}
\noindent
A Monte Carlo simulation (see \S\ref{sec:R_density}) suggests that
$\rho(\infty)$ is equal to $0.87411$ to five decimal places.

Theorem~\ref{maingenusonelocal} and Proposition~\ref{prop:R-density-bounds},
together with equation \eqref{prodformula}, imply the following.

\begin{theorem}\label{main}
When equations of the form~$(\ref{hypereq})$ with coefficients in~$\Z$
are ordered by height, a proportion of
\[ \rho = \rho(\infty)\frac{23087}{24528}\prod_{p>2} \left(
1 - \frac{4 p^{7} + 4 p^{6} + 2 p^{5} + p^{4} + 3 p^{3} + 2 p^{2} + 3 p + 3}{8(p + 1) (p^{2} + p + 1)  (p^{6} + p^{3} + 1)}
\right) \]
have points everywhere locally.
We have
\[
              0.759515 \le \rho \le  0.759663.
  \]
\end{theorem}
\noindent
The Monte Carlo simulation described above suggests that the value of $\rho$ is equal to  
$0.75965$ to five decimal places. 

One striking feature of Theorem~\ref{maingenusonelocal} is that, for
$p>2$, the quantities $\rho(p)$ are given by a fixed rational
function $R$ evaluated at $p$.
In a sequel to this
article we will show that this phenomenon continues for higher genus
hyperelliptic curves, provided that $p$ is sufficiently large compared
to the genus.

Our strategy for proving Theorem~\ref{maingenusonelocal} 
is a refinement of
that for testing solubility of a binary quartic form over $\Zp$
(equivalently, $\Qp$) as described, for example, in the work of Birch
and Swinnerton-Dyer~\cite{BSD}; the arguments 
are also related to
those in our earlier work on determining the density of locally
soluble quadratic forms (with Keating and Jones) in~\cite{BCFJK} and
ternary cubic forms in \cite{BCF-cubic}.  We consider the reductions
modulo~$p$; equations (\ref{hypereq}) whose reductions have smooth
$\Fp$-points are soluble by Hensel's lemma, while those that have no
$\Fp$-points are insoluble.  Finally, to determine the probabilities of
solubility in the much more difficult remaining cases, we develop certain
recursive formulae, involving these and other suitable related
probabilities, that allow us to solve for and obtain exact algebraic
expressions for the desired probabilities.

\bigskip

We may instead consider equations of the form
\begin{equation}
    z^2 + h(x,y)z = f(x,y), \label{eq:gen_bq}
\end{equation}
with $f(x,y)$ as in~\eqref{hypereq} and~$h(x,y)=lx^2+mxy+ny^2$.  For
these equations, as in \cite{CFS}, we get a more uniform result for
all primes~$p$, including $p=2$.  For each~$p$, we identify the space
of generalized binary quartics~\eqref{eq:gen_bq} over~$\Zp$ with the
space $\Zp^8$ equipped with its natural additive Haar measure.  Then
we prove the following theorem.

\begin{theorem}\label{thm:genbinquartic}
  For every prime~$p$, the density~$\rho'(p)$ of generalized binary
  quartics \eqref{eq:gen_bq} over $\Zp$ that are soluble over $\Qp$
  is $R(p)$, where $R$ is the rational function~\eqref{R1-formula}.
\end{theorem}
\noindent
Note that $\rho(2)=23087/24528\approx0.94125$ while
$\rho'(2)=1625/1752\approx0.92751$.

We define the \emph{height}~$H(h,f)$ of a generalized binary 
quartic~\eqref{eq:gen_bq} given by a pair of binary forms~$(h,f)$ by
\begin{equation}\label{def:Hgen}
H(h,f) :=\max\{|l|^2,|m|^2,|n|^2,|a|,|b|,|c|,|d|,|e|\},
\end{equation}
which generalizes the height~$H(f)$ defined for binary quartics.  As
before, we define the probability $\rho'(\infty)$ that an equation of the
form~\eqref{eq:gen_bq}, with real coefficients independently and
uniformly distributed in $[-1,1]$, has a real solution.
The
value of~$\rho'(\infty)$ is different from $\rho(\infty)$, and is
approximately~$0.873743$: see Section~\ref{sec:R_density}.

Finally, we let $\rho'$ denote the density of equations
(\ref{eq:gen_bq}) with integer coefficients, when ordered by their
height~$H(h,f)$, that are everywhere locally soluble; that is,
\begin{equation}\label{def:rho'}
\rho' = \lim_{X\to\infty} \frac{\#\{(h,f) \mid H(h,f) \le
  X,\ \text{(\ref{eq:gen_bq}) everywhere locally soluble}\}}
                    {\#\{(h,f) \mid H(h,f) \le X\}}.
\end{equation}
We now obtain a version of Theorem~\ref{main} for generalized binary
quartics.  To give the analogue of~\eqref{prodformula} also for
$\rho'$, we use a weighted version of~\cite[Lemma 1]{PS2}, proved by
the same methods in~\cite[Prop.~3.4]{CremonaSadek}, that takes into
account the weights in the definition of~$H(h,f)$.  We thus obtain the
following theorem.

\begin{theorem}\label{main-genbq}
When equations of the form~$(\ref{eq:gen_bq})$ with coefficients
in~$\Z$ are ordered by height, a proportion of
\[ \rho' = \rho'(\infty)\prod_{p} \left(
1 - \frac{4 p^{7} + 4 p^{6} + 2 p^{5} + p^{4} + 3 p^{3} + 2 p^{2} + 3
  p + 3}{8(p + 1) (p^{2} + p + 1) (p^{6} + p^{3} + 1)} \right) \] have
points everywhere locally.  We have
\[
    \rho' \approx 0.748248.  
\]
\end{theorem}

The results of this paper are used by the first author in
\cite{most-hyper} to prove that a positive proportion of equations of
the form~\eqref{hypereq} fail the Hasse Principle.  They are also used
by the third author, Ho, and Park in~\cite{FisherHoPark} to determine
the density of bidegree $(2,2)$-forms over~$\Z$ (which correspond to
genus one curves over~$\Q$ embedded in $\P^1\times\P^1$) that have
points everywhere locally.

The proof of Theorem~\ref{thm:genbinquartic} is given in
Section~\ref{sec:bq_density}.  After some preliminaries
in~\S\ref{subsec:notprelim} and~\S\ref{subsec:counts}, the main proof
is presented in~\S\ref{subsec:proofThm4}.  For odd primes~$p$,
Theorem~\ref{maingenusonelocal} then follows immediately, since by
completing the square we have $\rho(p) = \rho'(p)$. The modifications
required when $p = 2$ are described in \S\ref{subsec:p=2}. Finally, in
Section~\ref{sec:R_density}, we describe the methods we used to
estimate the probability $\rho(\infty)$ of solubility over the
reals, establishing Proposition~\ref{prop:R-density-bounds} and hence
Theorem~\ref{main}.

\section{The density of soluble generalized binary quartics over $\Zp$}
\label{sec:bq_density}
In this section, we determine the probability that a genus one curve
over $\Qp$, given by an equation in the general form
\eqref{eq:gen_bq} with coefficients in~$\Zp$, has a
$\Qp$-rational point.

A similar, slightly simpler, argument can be applied to the equation
\eqref{hypereq} over~$\Zp$ for odd primes~$p$, and yields exactly the same
probability; it is easy to see that this must be the case, using a straightforward
argument based on completing the square when $2$ is a unit.

\subsection{Notation and preliminaries}
\label{subsec:notprelim}
Let $h(x,y)=lx^2+mxy+ny^2$ and $f(x,y)= ax^4+bx^3y+cx^2y^2+dxy^3+ey^4$
be binary forms over~$\Zp$, and let
\[
F(x,y,z)=z^2+h(x,y)z-f(x,y).
\]
We refer to the pair $(h,f)$ or $F$ itself as a ``generalized binary
quartic''.  The polynomial $F$ is weighted homogeneous, where $x,y,z$
have weights $1,1,2$ respectively, and defines a curve~$\mC$ in
weighted projective space $\P(1,1,2)$ over~$\Qp$; this has genus one
provided that it is smooth.  We denote reduction modulo~$p$ by a bar,
so that $\hbar, \fbar \in \Fp[x,y]$ are (possibly zero) binary forms
over~$\Fp$.

For every $\Qp$-point~$(x:y:z)$, we may choose homogeneous
coordinates~$x,y,z\in\Zp$, not all in~$p\Zp$; then at least one of
$x,y$ is a unit.  In what follows, we will always choose such
primitive integral coordinates.

Our overall strategy is based on the observation that $\Qp$-rational
points reduce modulo~$p$ to $\Fp$-points on the reduced curve~$\Cbar$
and that smooth $\Fp$-points lift to $\Qp$-points by Hensel's lemma.
Thus if $\Cbar$ has smooth $\Fp$-points, then
$\mC(\Qp)\not=\emptyset$, while if $\Cbar(\Fp)=\emptyset$, then
$\mC(\Qp)=\emptyset$.  If all the $\Fp$-points are
singular, then we have to work harder: geometrically, we then blow up
the singular points; in our exposition, we will explicitly make
variable substitutions.  This will lead to a recursion, from which we
will then be able to solve for the various densities or probabilities of
solubility in different special configurations.

\subsection{Counts of generalized binary quartics over $\Fp$}
\label{subsec:counts}
We divide into cases according to the factorization of $\Fbar$ over
the algebraic closure~$\overline{\F}_p$ of~$\Fp$; clearly, this
occurs either over~$\Fp$ itself or over~$\F_{p^2}$, giving the following
four \emph{factorization types}:
\begin{enumerate}
\item $\Fbar$ absolutely irreducible;
\item $\Fbar$ has distinct factors over~$\Fp$, i.e.,
  $\Fbar=(z-s_1(x,y))(z-s_2(x,y))$ with $s_1,s_2\in\Fp[x,y]$
  distinct;
\item $\Fbar$ has conjugate factors over~$\F_{p^2}$, i.e.,
  $\Fbar=(z-s_1(x,y))(z-s_2(x,y))$ with
  $s_1,s_2\in\F_{p^2}[x,y]$ conjugate over $\Fp$;
\item $\Fbar$ has a repeated  factor over~$\Fp$, i.e.,
  $\Fbar=(z-s(x,y))^2$ with $s\in\Fp[x,y]$.
\end{enumerate}
In the following lemma we give the counts of how many of the $p^8$
pairs $(\hbar,\fbar)$ fall into each of these cases, and how many
there are that also satisfy the side condition
\[
    z^2 + lz - a \qquad\text{is irreducible over $\Fp$}  \tag{$*$}
\]
which will occur later, and will be referred to as ``condition $(*)$''.

\begin{lemma}
\label{lem:gbq_counts}
The numbers of generalized binary quartics over~$\Fp$ with each
factorization type are given in the following table, as well as the
same counts for those satisfying condition $(*)$: \em
\[ \begin{array}{|l|c|c|}
\hline
\text{Factorization type} & \text{All} & \text{Satisfying $(*)$}\\ \hline\hline
\text{1. Absolutely irreducible}  &  p^6(p^2-1) & \frac{1}{2}p^5(p^2-1)(p-1)\\
\text{2. Distinct factors over $\Fp$}  &  \frac{1}{2}p^3(p^3-1) & 0 \\
\text{3. Conjugate factors over $\F_{p^2}$}  & \frac{1}{2}p^3(p^3-1) & \frac{1}{2}p^5(p-1)\\
\text{4. Repeated factor over $\Fp$}  &  p^3 & 0\\ \hline
\text{Total} & p^8 & \frac{1}{2}p^7(p-1)\\ \hline
\end{array} \]
\end{lemma}
Let $\xi_i$ (respectively, $\xi_i^*$) denote the probability that a
generalized binary quartic~$F$ over $\Fp$ (respectively, a
generalized binary quartic~$F$ satisfying $(*)$) has factorization
type~$i$ for $i=1,2,3,4$.  Then Lemma~\ref{lem:gbq_counts} implies the
following.
\begin{corollary}
\label{cor:gbq_probs}
The probabilities $\xi_i$, $\xi_i^*$ are given as follows:
\[
\begin{array}{|c|c|c|}\hline
  i & \xi_i  & \xi_i^* \\
  \hline\hline
  1 & (p^2-1)/p^2 & (p^2-1)/p^2 \\
  2 & \frac{1}{2}(p^3-1)/p^5 &  0 \\
  3 & \frac{1}{2}(p^3-1)/p^5 & 1/p^2 \\
  4 & 1/p^5 & 0 \\ \hline
\end{array}
\]
\end{corollary}

Define a binary form~$f(x,y)$ to be \emph{monic} if $f(1,0)=1$.  We
will need the counts of binary quartics (up to scaling), and the
number of monic binary quartics, over~$\Fp$ with certain
factorization patterns over~$\Fp$, distinguished by the number of
roots in $\P^1(\Fp)$ with various multiplicities: (0) none, (1) at
least one simple root, (2) a double and no simple roots, (3) two
double roots, or (4) a quadruple root.  An elementary computation
yields the following.
\begin{lemma}
\label{lem:q_counts}
The numbers of nonzero binary quartics over~$\Fp$ (up to scaling by $\Fp^\times$)
with each factorization type are given in the following table, as well
as the same for monic quartics: \em
\[ \begin{array}{|l|c|c|c|}\hline
\text{Factorization type} & \text{Binary quartics mod } \Fp^\times & \text{Monic quartics} \\ \hline \hline
\text{0. No roots} & \frac{1}{8}p(p-1)(3p^2+p+2) & \frac{1}{8}p(p-1)(3p^2+p+2) \\
\text{1. Simple root}  & \frac{1}{8}p(p+1)(5p^2+p+2) & \frac{1}{8}p(p-1)(5p^2+3p+2)  \\
\text{2. One double and no simple root} & \frac{1}{2}p(p^2-1) & \frac{1}{2}p^2(p-1) \\
\text{3. Two double roots} & \frac{1}{2}p(p+1) & \frac{1}{2}p(p-1) \\
\text{4. Quadruple root} & p+1 & p \\ \hline
\text{Total} & p^4 + p^3 + p^2 + p + 1 & p^4\\ \hline
\end{array} \]
\end{lemma}
Let $\eta_i$ (respectively, $\eta_i'$) denote the probability that a
nonzero binary quartic form~$f$ over $\Fp$ (respectively, a monic binary
quartic form~$f$) has factorization type~$i$ for $i=0,1,2,3,4$.  Then
Lemma~\ref{lem:q_counts} implies the following.
\begin{corollary}\label{cor:q_probs}
The probabilities $\eta_i$, $\eta_i'$ are given as follows:
\[
\begin{array}{|c|c|c|}\hline
  i &\eta_i  & \eta_i' \\
 \hline\hline
  0 & \frac{1}{8}{p(p-1)^2(3p^2+p+2)}/{(p^5-1)} & \frac{1}{8}{(p-1)(3p^2+p+2)}/{p^3}\\
  1 & \frac{1}{8}{p(p^2-1)(5p^2+p+2)}/{(p^5-1)} & \frac{1}{8}{(p-1)(5p^2+3p+2)}/{p^3}\\
  2 & \frac{1}{2}{p(p-1)(p^2-1)}/{(p^5-1)} & \frac{1}{2}{(p-1)}/{p^2}\\
  3 & \frac{1}{2}{p(p^2-1)}/{(p^5-1)} & \frac{1}{2}{(p-1)}/{p^3}\\
  4 & {(p^2-1)}/{(p^5-1)} & {1}/{p^3}\\ \hline
\end{array}
\]
\end{corollary}

\subsection{Proof of Theorem \ref{thm:genbinquartic}}
\label{subsec:proofThm4}
Fix a prime $p$, and let $\rho=\rho'(p)$ be the probability that the
equation~(\ref{eq:gen_bq}) with coefficients in~$\Zp$ is
$\Qp$-soluble.  Let $\sigma_i$ denote the probability of solubility
of a generalized binary quartic that has factorization type~$i$.  Then
\begin{equation}\label{eqn:rho-from-rho4}
   \rho = \sum_{i=1}^{4}\xi_i\sigma_i,
\end{equation}
where the $\xi_i$ are as in Corollary~\ref{cor:gbq_probs}.

Similarly, let $\rho^*$ be the probability that a generalized
binary quartic~(\ref{eq:gen_bq}) that satisfies condition~$(*)$ is
$\Qp$-soluble.  For $i=1$ and $i=3$, let $\sigma_i^*$ denote the
probability of solubility of a generalized binary quartic that has
factorization type~$i$ and satisfies~$(*)$. (We do not define
$\sigma_2^*$ or $\sigma_4^*$.)  Then
\[
   \rho^* = \xi_1^*\sigma_1^* +  \xi_3^*\sigma_3^*.
\]

To compute~$\rho$, we evaluate each~$\sigma_i$; in so doing, we will
also need the values of~$\sigma_1^*$, $\sigma_3^*$ and~$\rho^*$. We
evaluate $\sigma_1$, $\sigma_1^*$, and $\sigma_2$ in
\S\ref{subsubsec:sigma12}, then~$\sigma_3$ and $\sigma_3^*$ in
\S\ref{subsubsec:sigma3}, and~$\sigma_4$ in \S\ref{subsubsec:sigma4}.
The latter is expressed in terms of $\rho$ and additional probabilities~$\tau_i$
for $0\le i\le 4$, defined and evaluated in~\S\ref{subsubsec:sigma4}
and~\S\ref{subsubsec:tau4}.  We then solve the resulting recursion for
$\rho$ in~\S\ref{subsubsec:conclusion}.

\subsubsection{Evaluation of $\sigma_1$, $\sigma_1^*$, and $\sigma_2$}
\label{subsubsec:sigma12}

The first two cases, where the reduction~$\Fbar$ is either absolutely
irreducible, or has distinct factors over~$\Fp$, are straightforward.
First, both $\sigma_1$ and~$\sigma_1^*$ are probabilities of $\Q_p$-solubility of a
generalized binary quartic~$F$ over~$\Zp$ whose reduction modulo~$p$
is absolutely irreducible. Such curves
always have smooth $\Fp$-points; this would not necessarily be the
case for hyperelliptic curves of genus~$g\ge2$, which will be treated
in a sequel to this paper.

\begin{proposition}
\label{prop:gbq_absirred}
Every generalized binary quartic over~$\Zp$ whose reduction
modulo~$p$ is absolutely irreducible has a $\Qp$-rational point; that
is, $\sigma_1=\sigma_1^*=1$.
\end{proposition}
\begin{proof} The curve $\Cbar$ over~$\Fp$ defined by a generalized binary quartic
over~$\Zp$ whose reduction modulo~$p$ is absolutely irreducible has
arithmetic genus~$1$.  If it is smooth, then it has genus~$1$ and thus at
least one~$\Fp$-point by the Hasse bounds.  Otherwise, since it is
geometrically irreducible, its normalization is a smooth curve of
genus zero.  Since the genus drops by at least~$1$ for each singular
point, there must be exactly one singularity, and its multiplicity
must be~$2$.  Now the normalization has $p+1$ points, of which at most
two lie over the singular point of~$\Cbar$, so $\Cbar$ has at least
$p-1$ smooth points.  Thus, in all cases, $\Cbar$ has at least one
smooth point over~$\Fp$, which lifts to a $\Qp$-point.
\end{proof}

\begin{proposition}
\label{prop:gbq_split}
Every generalized binary quartic over~$\Zp$ whose reduction
modulo~$p$ splits into two distinct factors has a $\Qp$-rational
point, so $\sigma_2=1$.
\end{proposition}
\begin{proof}
We have $\Fbar = (z-s_1(x,y))(z-s_2(x,y))$, where $s_1,s_2\in\Fp[x,y]$
are distinct binary quadratic forms.  Each of the curves $z=s_i(x,y)$
has $p+1$ points over~$\Fp$, and since $s_1\not=s_2$, these two
curves intersect in at most~$2$ points. Hence, for all~$p$, there are
smooth $\Fp$-points, which lift to $\Qp$-points.
\end{proof}

\subsubsection{Some lemmas}
\label{subsubsec:lemmas}

The following lemma expresses a general principle that allows us to
make a change of coordinates in the arguments that follow.

\begin{lemma} \label{newlemma}
Let $\Phi$ and~$\Phi'$ be subsets of~$\Z_p[x_1, \ldots, x_n]$, each
defined by letting $m$ of the coefficients range over an open subset
of $\Z_p^m$ and fixing the remaining coefficients.  Suppose that
$\Phi'$ is the image of $\Phi$ under an affine linear substitution
\begin{equation}
\label{affine-linear}
\begin{pmatrix} x_1 \\ \vdots \\ x_n \end{pmatrix} \mapsto A
\begin{pmatrix} x_1 \\ \vdots \\ x_n \end{pmatrix} + b,
\end{equation}
where $A \in \GL_n(\Z_p)$ and $b \in \Z_p^n$.
Then the probability that a random polynomial in $\Phi$ is soluble
over $\Z_p$ is the same as the probability that a random polynomial in
$\Phi'$ is soluble over $\Z_p$.
\end{lemma}
\begin{proof}
The effect of the transformation~\eqref{affine-linear} on the $m$
coefficients that vary is both affine linear and invertible mod $p$.
It therefore preserves the measure. Clearly solubility over $\Z_p$
is also preserved.
\end{proof}

As an example of the use of this lemma, we consider the set of monic
quadratics $g(x)=x^2+bx+c$ with coefficients in~$\Zp$, which we
identify with~$\Zp^2$ in the usual way. The lemma shows that the
probability $g$ has a $\Zp$-root, given that $g$ has a double root mod
$p$ at $x \equiv r \pmod{p}$, is independent of $r$. Thus in working
out the probability that $g$ has a $\Z_p$-root, given that $g$ has a
double root mod $p$, we may assume without loss of generality
that the double root is at $x \equiv 0 \pmod{p}$.

\medskip

The following two lemmas will be used to help evaluate $\sigma_3$
and~$\sigma_3^*$ in the next section.
The coefficients of a generalized binary quartic
are labelled as described in Section~\ref{subsec:notprelim}.
\begin{lemma}
\label{lem:two-lemmas-1}
Let $l,a,b \in p\Zp$ and $m,c\in\Zp$ be fixed, subject to the
condition that $z^2+mz-c$ is irreducible over~$\Fp$.  Let $\oldbeta$ be
the probability of the existence of $x,z\in\Zp$ with $F(x,1,z)=0$
given that $n,d,e\in\Zp$ (that is, $\oldbeta$ is the density of
$(n,d,e)\in\Zp^3$ for which such a solution exists), and $\oldalpha$ the
probability of such a solution given that $n,d,e\in p\Zp$. Then
\[
    \oldbeta = \dfrac{p}{p+1}
    \qquad\text{and}\qquad 
    \oldalpha = \dfrac{1}{p+1}.
\]
\end{lemma}
\begin{proof}
We will show that $\oldbeta=(1-1/p)+(1/p)\oldalpha$ and $\oldalpha=(1/p)\oldbeta$,
from which their values follow.  Note that while, \textit{a priori},
$\oldbeta$ and~$\oldalpha$ might depend on~$(l,a,b)\in p\Zp^3$ and~$(m,c)\in
\Zp^2$, we will see in the proof that this is not the case.

First, assume only that $n,d,e\in\Zp$. We have~$F(x,1,z) \equiv
z^2+(mx+n)z-(cx^2+dx+e)$, which defines a conic over~$\Fp$. The side
condition that $z^2+mz-c$ is irreducible over~$\Fp$ implies that
there are no $\Fp$-rational points at infinity.  If this conic is
smooth then it has $\Fp$-points; the probability of this is $1-1/p$,
since the discriminant of the conic is $e(m^2+4c)-(d^2+dmn-cn^2)$ and
$m^2+4c\not=0$, so for each pair $n,d$ there is a unique $e\pmod{p}$
for which the discriminant vanishes.  If the conic is singular, then
the singular point is the only $\Fp$-rational point (by the side
condition), and without loss of generality (see Lemma~\ref{newlemma})
we may suppose the singular
point to be at $(0,0)\pmod{p}$, for which the probability of
solubility is~$\oldalpha$.  This establishes the first equation.

Next suppose that $n,d,e\in p\Zp$. Then we have $F(x,1,z) \equiv
z^2+mxz-cx^2$, whose only zero over $\Fp$ is (by the side condition)
$(x,z)\equiv(0,0)$. The equation $F(x,1,z)=0$ implies $p^2\mid e$, so
with probability $1-1/p$ we have no solutions; otherwise we may
replace the variables $x,z$ by $px,pz$ and divide through by $p^2$,
leading back to the first case.  This establishes the second equation.

Note that in this reduction step the values of $c$ and~$m$ are
unchanged.  The coefficients $l,a,b$ become more divisible by~$p$ but
our first argument is unchanged, being independent of these values,
provided only that they have positive valuation.
\end{proof}

\begin{lemma}
\label{two-lemmas-2}
Let $l,m,c \in p\Z_p$ and $a,b\in p^2\Z_p$ be fixed, with
$v(c)=1$. Let $\alpha'$ be the probability of the existence of
$x,z\in\Z_p$ with $F(x,1,z)=0$ given that $n,e\in\Z_p$ and~$d\in p\Z_p$
(that is, $\alpha'$ is the density of~$(n,d,e)\in\Zp\times
p\Zp\times\Zp$ for which such a solution exists), and $\beta'$ the
probability of such a solution given that $n,d,e\in p\Z_p$. Then
\[
    \alpha' = \beta' = \dfrac{1}{2}.
\]
\end{lemma}
\begin{proof}
We will show that $\alpha'=\frac{1}{2}(1-1/p)+(1/p)\beta'$ and
$\beta'=\frac{1}{2}(1-1/p)+(1/p)\alpha'$, from which the result
follows.  Note that while, \textit{a priori}, $\alpha'$ and~$\beta'$
might depend on~$(l,m,c)\in p\Zp^3$ and~$(a,b)\in p^2\Zp^2$, we will
see in the proof that this is not the case.

First let $n,d,e\in p\Z_p$ and consider the quadratic
$\frac{1}{p}(cx^2+dx+e)$ over~$\Fp$.  If it has distinct roots, then
we can lift these to obtain $\Q_p$-points with $z=0$; this has
probability~$\frac{1}{2}(1-1/p)$.  If it is irreducible
(probability~$\frac{1}{2}(1-1/p)$) then there are no solutions, while
if it has repeated roots (probability~$1/p$) then after shifting the
root to $x\equiv0$ (see Lemma~\ref{newlemma}), replacing $x,z$ by
$px,pz$ and rescaling the equation we recover the original situation
except that now we only have $n,e\in\Z_p$ with $d\in p\Z_p$.  (The
valuations of $l,a,b$ have increased, but everything here only depends
on them lying in $p\Z_p$ or $p^2\Z_p$ as specified in the statement of
the lemma.)  This gives $\beta'=\frac{1}{2}(1-1/p)+(1/p)\alpha'$.

Now with $n,e\in\Z_p$ and $d\in p\Z_p$, we have
$\Fbar(x,1,z)\equiv z^2+nz-e$.  The equation is soluble if this
quadratic has distinct roots, and insoluble if it has no roots; if it
has a double root then shifting the root to~$z\equiv0$ leads us back
to the first case.  Hence
$\alpha'=\frac{1}{2}(1-1/p)+(1/p)\beta'$, as required.
\end{proof}

\subsubsection{Evaluation of $\sigma_3$ and $\sigma_3^*$}
\label{subsubsec:sigma3}

Now we consider generalized binary quartics~$F$ whose reduction
modulo~$p$ factors as $\Fbar = (z-s_1(x,y)(z-s_2(x,y))$, where~$s_1$
and~$s_2$ are conjugate binary quadratics over~$\F_{p^2}$.  Let
$\omega\in\F_{p^2}\setminus\Fp$ and denote by $\overline{\omega}$ its
Galois conjugate; then we may write $s_1(x,y)=r(x,y)+\omega s(x,y)$
and~$s_2=r(x,y)+\overline{\omega}s(x,y)$, where $r,s$ are binary
quadratic forms over~$\Fp$.  Replacing $z$ by $z+r(x,y)$, we may
assume without loss of generality that $r=0$, so now $\Fbar =
(z-\omega s(x,y))(z-\overline\omega s(x,y))$.
The only $\Fp$-points are those with $z\equiv s(x,y)\equiv0\pmod{p}$,
which are singular.  The probability of solubility now depends on the
factorization of~$s$ over~$\Fp$.  Under condition~$(*)$, the leading
coefficient of~$s$ must be non-zero, and the only difference between
the two cases ($\sigma_3$ and $\sigma_3^*$) arises from the different
probabilities of each factorization pattern occurring, depending on
whether $s$ is an arbitrary binary quadratic form over~$\Fp$, or is
restricted to those whose leading coefficient is non-zero.

In the case where~$s$ has distinct roots modulo~$p$, we must take care
to show that the probabilities that the two singular $\Fp$-points
lift are independent.

\begin{lemma}
\label{lem1}
Suppose that $F$ is a generalized binary quartic over~$\Zp$ whose
reduction modulo~$p$ factors over~$\F_{p^2}$ as $\Fbar = (z-\omega
s)(z-\overline\omega s)$, with $s\in\Fp[x,y]$ a binary quadratic form
having distinct roots over~$\Fp$.  Then the curve~$\Cbar$ over~$\Fp$
defined by~$F$ has two $\Fp$-points, both singular, and the
probability that at least one of these points lifts to a $\Qp$-point
is $(2p+1)/(p+1)^2$.
\end{lemma}
\begin{proof}
Moving the roots of~$s$ to $(0:1)$ and $(1:0)$, we may assume
without loss of generality that $s=xy$.  Let $l,m,n,a,b,c,d,e\in \Zp$
be the coefficients of~$F$.  Then we have $l,n,a,b,d,e\in p\Zp$, and
$z^2+mz-c$, being the minimal polynomial of~$\omega$, is irreducible
over~$\Fp$.  The only~$\Fp$-points are the two singular points,
$(0:1:0)$ and~$(1:0:0)$.

The probability that~$(0:1:0)$ lifts to a $\Qp$-point
is~$\oldalpha=1/(p+1)$, since we are in the situation of
Lemma~\ref{lem:two-lemmas-1} with~$n,d,e\in p\Zp$.  By symmetry, the probability
that~$(1:0:0)$ lifts to a $\Qp$-point is also~$\oldalpha$.  Provided
that these probabilities are independent, the probability that at
least one of the points lifts is $1-(1-\oldalpha)^2 = (2p+1)/(p+1)^2$, as
stated.

If $(1:0:0)$ lifts to a $\Qp$-point then, by substitutions of the form
$y \mapsto y + rx$ and $z \mapsto z + s x^2$ where $r,s \in p \Zp$, we
may assume that the $\Qp$-point is $(1:0:0)$, and so $a = 0$. By the
proof of Lemma~\ref{newlemma}, the effect of such a transformation on
$(n,d,e) \in p \Zp^3$ is measure preserving. Then applying
Lemma~\ref{lem:two-lemmas-1} with $a = 0$ shows that the probability
$(0:1:0)$ lifts to a $\Qp$-point is still~$\beta$.
Therefore the two probabilities of lifting are
indeed independent, as claimed.
\end{proof}

In the case where $s$ has a double root modulo~$p$, our argument
expresses the probability of solubility in terms of~$\rho^*$.

\begin{lemma}
\label{lemB}
Suppose that $F$ is a generalized binary quartic over~$\Zp$ whose
reduction modulo~$p$ factors over~$\F_{p^2}$ as a product of two
conjugate factors, in the form $\Fbar=(z-\omega s)(z-\overline \omega
s)$, where $s$ is a binary quadratic form over~$\Fp$ with a double
root over~$\Fp$.  Then the curve~$\Cbar$ over~$\Fp$ defined
by~$\Fbar$ has only one $\Fp$-point, and the probability~$\lambda$
that this point lifts to a $\Qp$-point satisfies
\[
\lambda =  \frac{2p^5-p^4+p^2-2p}{2p^6} + \frac{\rho^*}{p^6}.
\]
\end{lemma}
\begin{proof}
Moving the repeated root of~$s$ to $(0:1)$, we may assume without
loss of generality that~$s=x^2$.  Let $l,m,n,a,b,c,d,e\in \Zp$ be the
coefficients of~$F$.  Then we have $m,n,b,c,d,e\in p\Zp$, and
$z^2+lz-a$, being the minimal polynomial of~$\omega$, is irreducible
modulo~$p$, so that condition~$(*)$ holds.  The only $\Fp$-point is
the singular point $(0:1:0)$, and we need to determine the probability
$\lambda$ that this point lifts.

The situation is as shown in the first row of the following table; 
here~$\lambda_i$ denotes the probability of solubility given both
condition~$(*)$ and that the valuations of $l,\dots,e$ satisfy the
conditions in row~$i$; in particular, $\lambda=\lambda_1$.  The values
of~$l$ and~$a$ change from row to row but return to their original
values by line~7; condition~$(*)$ refers always to the original values
of~$l$ and~$a$.
\[ \begin{array}{rlcccccccccc}
& & & l & m & n & & a & b & c & d & e \\
\lambda = &\lambda_1 = \tfrac{1}{p} \lambda_2 & &
  \ge0 & \ge 1 & \ge 1 & &
  \ge0 & \ge 1 & \ge 1 & \ge 1 & \ge 1 \\
&\lambda_2 = (1-\tfrac{1}{p}) +  \tfrac{1}{p} \lambda_3 & &
  \ge1 & \ge 1 & \ge 0 & &
  \ge2 & \ge 2 & \ge 1 & \ge 0 & \ge 0 \\
&\lambda_3 = \tfrac{1}{2}(1-\tfrac{1}{p}) +  \tfrac{1}{p} \lambda_4 & &
  \ge1 & \ge 1 & \ge 0 & &
  \ge2 & \ge 2 & \ge 1 & \ge 1 & \ge 0 \\
&\lambda_4 = \tfrac{1}{2}(1-\tfrac{1}{p}) +  \tfrac{1}{p} \lambda_5 & &
  \ge1 & \ge 1 & \ge 1 & &
  \ge2 & \ge 2 & \ge 1 & \ge 1 & \ge 1 \\
&\lambda_5 = (1-\tfrac{1}{p}) +  \tfrac{1}{p} \lambda_6 & &
  \ge1 & \ge 1 & \ge 1 & &
  \ge2 & \ge 2 & \ge 2 & \ge 1 & \ge 1 \\
&\lambda_6 = \tfrac{1}{p} \lambda_7 & &
  \ge1 & \ge 1 & \ge 1 & &
  \ge2 & \ge 2 & \ge 2 & \ge 2 & \ge 1 \\
&\lambda_7 = \rho^* & &
  \ge0 & \ge 0 & \ge 0 & &
  \ge0 & \ge 0 & \ge 0 & \ge 0 & \ge 0 \\
\end{array} \]
Assuming the correctness of this table, this gives the desired
expression for $\lambda$ in terms of $\rho^*$.

Now we justify the reduction steps and the relations between
successive $\lambda_i$.  At each step, we determine whether there is a
smooth solution which lifts, or there is no solution, or we continue
to the next line; the probability of continuing to the next line is
always~$1/p$.
\begin{enumerate}
  \item In line 1, condition~$(*)$ implies that every solution has
    $x,z\in p\Zp$; so there are no solutions unless $v(e)\ge2$, which
    has probability~$1/p$.  Replacing $x$ by~$px$ and $z$ by~$pz$, and
    then dividing by~$p^2$, leads to line~2.
  \item The reduced equation is now $z^2+nz\equiv dx+e$, which is
    linear in~$x$.  With probability~$1-1/p$, we have~$v(d)=0$ and a
    solution exists. Otherwise, with probability $1/p$, we have
    $v(d)\ge1$, leading to line~3.
  \item The reduced equation is now $z^2+nz\equiv e$.  If this
    quadratic is irreducible over~$\Fp$, which happens with
    probability $\frac{1}{2}(1-1/p)$, then there are no solutions; 
    if it splits (which happens with the same probability), then there are
    smooth solutions which lift to a $\Qp$-point with~$x=0$. Finally,
    with probability $1/p$, the quadratic has a double root, which we
    shift to $z\equiv0$, leading to line~4.
  \item With probability $1-1/p$ we have $v(c)=1$; by
    Lemma~\ref{two-lemmas-2} the equation is soluble with probability $1/2$.
    Otherwise, with probability $1/p$, we have $v(c)\ge2$, leading to
    line~5.
  \item With probability $1-1/p$ we have $v(d)=1$; then the quartic
    $\frac{1}{p}f$ has a simple root modulo~$p$ (at $x\equiv0$), so we can lift
    it to a $\Qp$-point with $z=0$.  Otherwise, with probability $1/p$,
    we have $v(d)\ge2$, leading to line~6.
  \item As in line~1, we now have no solutions unless $v(e)\ge2$,
    which happens with probability~$1/p$.  At this point $p$ divides
    $l,m$, and~$n$, while $p^2$ divides $a,b,c,d$, and~$e$, so we may
    scale the equation to obtain line~7.
  \item We now have a generalized binary quartic satisfying
    no conditions other than~$(*)$, because the coefficients $l$
    and~$a$ at the end are the same as at the start, so the
    probability of solubility is~$\rho^*$.
\end{enumerate}
\end{proof}

We can now evaluate both~$\sigma_3$ and $\sigma_3^*$. The rational
functions recorded here, and in the remainder of
Section~\ref{sec:bq_density}, are not necessarily written in lowest
terms, since we prefer to give expressions where the denominator
is as simple as possible.

\begin{proposition}
\label{prop:gbq_conj}
A generalized binary quartic~$F$ over~$\Zp$ whose reduction
modulo~$p$ factors into two conjugate factors over~$\F_{p^2}$ has a
$\Qp$-rational point with probability
\[  \sigma_3 = \frac{(p-1)^2(2p^9 + 3p^8 + 5p^7 + 3p^6 + 5p^5 + 3p^4 +
                                                 4p^3 + 5p^2 + 4p + 1)}
        {2(p^3-1)(p^9-1)}. \]
For such a generalized binary quartic that also satisfies
condition~$(*)$, the probability of having a $\Qp$-rational point is
\[  \sigma_3^* = \frac{(p-1)(2 p^9 + 3 p^8 + 5 p^7 + 5 p^6 + 5 p^5 + 5 p^4 + 4 p^3 + 6 p^2 + 6 p + 2)}
        {2(p+1)^2(p^9-1)}. \]
\end{proposition}

\begin{proof}
Recall that we may assume without loss of generality that $\Fbar =
(z-\omega s(x,y))(z-\overline\omega s(x,y))$, where $s$ is a binary
quadratic form over~$\Fp$, and that the only $\Fp$-points are those
with $z\equiv s(x,y)\equiv0\pmod{p}$, which are singular.  We divide
into cases according to the factorization of~$s$ over~$\Fp$.

\begin{itemize}
\item If $s$ irreducible over $\Fp$, then there are no $\Fp$-points
  and the probability of solubility is~$0$.  This occurs with
  probability $\frac{1}{2}p(p-1)^2/(p^3-1)$ in general and with probability 
  $\frac{1}{2}(p-1)/p$ under condition~$(*)$.
\item The binary quadratic form $s$ splits into distinct factors over
  $\Fp$ with probability $\frac{1}{2}p(p^2-1)/(p^3-1)$ in general, or
  $\frac{1}{2}(p-1)/p$ under~$(*)$.  In this case there are two $\Fp$-points,
  one for each root of~$s$, and by Lemma~\ref{lem1}, the probability
  that at least one of these two $\Fp$-points lifts to~$\Qp$ is
  $(2p+1)/(p+1)^2$.
\item The binary quadratic form $s$ has a repeated factor over $\Fp$
  with probability $(p^2-1)/(p^3-1)$ in general, or $1/p$ under~$(*)$.
  In this case there is only one $\Fp$-point, corresponding to the
  unique root of~$s$, and Lemma~\ref{lemB} expresses the
  probability~$\lambda$ that it lifts to a $\Qp$-point in terms of
  $\rho^*$.
\end{itemize}

Combining the cases, we find that
\begin{equation}
  \label{eqn:sigma3-lambda}
  \sigma_3 = \frac{p(p^2-1)}{2(p^3-1)}\cdot\frac{(2p+1)}{(p+1)^2} +
  \frac{p^2-1}{p^3-1}\cdot\lambda
\end{equation}
while
\begin{equation}
  \label{eqn:sigma3*-lambda}
  \sigma_3^* = \frac{(p-1)}{2p}\cdot\frac{(2p+1)}{(p+1)^2} +
  \frac{1}{p}\cdot\lambda.
\end{equation}

Recall also that
\begin{equation}
  \label{eqn:sigma3*-rho*}
\rho^* = \xi_1^*\sigma_1^* + \xi_3^*\sigma_3^*
= \xi_1^* + \xi_3^*\sigma_3^*,
\end{equation}
(since $\sigma_1^*=1$ by Proposition~\ref{prop:gbq_absirred}), and the
values of $\xi_1^*$ and $\xi_3^*$ are given in
Corollary~\ref{cor:gbq_probs}.  The linear equations
\eqref{eqn:sigma3*-lambda} and~\eqref{eqn:sigma3*-rho*}, together with
Lemma~\ref{lemB}, can now be solved for~$\lambda$, $\sigma_3^*$ and
$\rho^*$, and finally equation~\eqref{eqn:sigma3-lambda}
gives~$\sigma_3$.
\end{proof}

\begin{corollary}
The probability that a generalized binary quartic~(\ref{eq:gen_bq})
that satisfies condition~$(*)$ is $\Qp$-soluble is
\[ \rho^* = \frac{p(p-1)(2 p^{9} + 6 p^{8} + 6 p^{7} + 4 p^{6} + 3 p^{5} + 5 p^{4} + 5 p^{3} + 5 p^{2} + 5 p + 2)}
               {2(p+1)^2(p^9-1)}. \]
\end{corollary}

For reference, we also record here the value of~$\lambda$ (defined in
Lemma~\ref{lemB}):
\[
\lambda = \frac{ 2 p^{10} + 3 p^{9} - p^{5} + 2 p^{4} - 2 p^{2} - 3 p - 1}{2(p+1)^2(p^9-1)}.
\]

\subsubsection{Evaluation of $\sigma_4$}
\label{subsubsec:sigma4}

We now evaluate $\sigma_4$, the probability of solubility given that
the generalized binary quartic~$F$ has repeated factors over~$\Fp$,
i.e., $F\equiv(z-s)^2\pmod{p}$ for some binary quadratic form~$s$.
Replacing $z$ by $z+s$ does not change densities, so we may assume
that $s\equiv0\pmod{p}$, so that $F\equiv z^2\pmod{p}$ or,
equivalently, that all eight coefficients of~$F$ lie in~$p\Zp$.

Now all solutions $(x:y:z)$ must have $z\in p\Zp$ and
$f(x,y)\equiv0\pmod{p^2}$.  Writing $f_1=\frac{1}{p}f$, we see that each solution
satisfies $f_1(x,y)\equiv z\equiv0\pmod{p}$.  We divide into cases
according to the factorization of $\fpbar$, the reduction of~$f_1$
modulo~$p$.

If $\fpbar=0$, all coefficients of~$f$ are divisible by~$p^2$, and we
may replace $z$ by~$pz$ and divide through by~$p^2$ to obtain an
arbitrary generalized binary quartic over~$\Zp$.  The probability of
this is $1/p^5$, and the probability of solubility in this
case is just~$\rho$.  Otherwise, for $0\le i\le4$, let $\tau_i$ be the
probability of solubility given each possible factorization pattern
for $\fpbar$, with the cases numbered as in Lemma~\ref{lem:q_counts}.
Then
\begin{equation}\label{eqn:rho4-from-rho}
   \sigma_4 = \frac{1}{p^5}\rho +
   \left(1-\frac{1}{p^5}\right)\sum_{i=0}^4\eta_i\tau_i,
\end{equation}
where the $\eta_i$ are as in Corollary~\ref{cor:q_probs}.

We also let $\sigma_4'$ be the probability that~(\ref{eq:gen_bq}) is
$\Qp$-soluble assuming again that all coefficients lie in~$p \Zp$
and also that~$v_p(a)=1$.  Then
\[
   \sigma_4' = \sum_{i=0}^{4}\eta_i' \tau_i,
\]
where the $\eta_i'$ are also as in Corollary~\ref{cor:q_probs}.

We will evaluate each~$\tau_i$; then (\ref{eqn:rho-from-rho4}) and
(\ref{eqn:rho4-from-rho}) will give two linear equations relating~$\rho$
and~$\sigma_4$, from which their values will then be uniquely determined.

The first two cases are easy.
\begin{proposition}
  \label{prop:tau01}
We have $\tau_0=0$ and $\tau_1=1$.
\end{proposition}
\begin{proof}
If $\fpbar$ has no roots in $\Fp$, then $f(x,y)\equiv0\pmod{p^2}$ has
no solutions, giving $\tau_0=0$.

If $\fpbar$ has a simple root in $\Fp$, it lifts to a root in $\Qp$,
giving a $\Qp$-point with $z=0$, so $\tau_1=1$.
\end{proof}

The next two cases are handled using Lemma~\ref{two-lemmas-2}.

\begin{proposition}
  \label{prop:tau23}
We have $\tau_2=1/2$ and $\tau_3=3/4$.
\end{proposition}

\begin{proof}
Each double root of $\fpbar$ over~$\Fp$ lifts to a $\Qp$-point with
probability~$\alpha'=1/2$; indeed, without loss of generality, the double
root is at $(x:y)=(0:1)$, so after replacing $x$ by~$px$ and $z$
by~$pz$, and rescaling the equation, we may apply Lemma~\ref{two-lemmas-2}.

If $\fpbar$ has one double root over~$\Fp$, and no simple roots, then
every $\Qp$-point must arise from lifting the double root, giving
$\tau_2=1/2$.

If $\fpbar$ has two double roots over~$\Fp$, we may assume that they
are $(x:y)=(0:1)$ and~$(1:0)$.  Each lifts to a $\Qp$-point with
probability~$1/2$ by Lemma~\ref{two-lemmas-2}, and examination of the proof of
Lemma~\ref{two-lemmas-2} shows that (just as in Lemma~\ref{lem1}) the two
probabilities are independent.  Hence $\tau_3=1-(1-1/2)^2=3/4$.
\end{proof}

\subsubsection{Evaluation of $\tau_4$}
\label{subsubsec:tau4}

Finally, we consider the case where $\fpbar$ has a quadruple root.
\begin{proposition}
\label{prop:tau4}
 We have $\tau_4=\dfrac{4 p^{10} + 8 p^9 - 4 p^8 + 4 p^6 - 3 p^4 + p^3 - 5 p -
    5}{8 (p+1) (p^9-1)}$.
\end{proposition}

\begin{proof}
Moving the quadruple root of $\fpbar$ to~$0$, without
loss of generality we have $\frac{1}{p}f\equiv a_1x^4\pmod{p}$ where
$a_1=a/p\not\equiv0\pmod{p}$.

Every primitive $\Qp$-point~$(x:y:z)$ must satisfy $x\equiv
z\equiv0\pmod{p}$ and $y\not\equiv0\pmod{p}$, so we replace $x,z$
by~$px,pz$ and divide through by~$p^2$; also, without loss of
generality, we may assume that $y=1$.  This leads to the situation
indicated in the first row of the following table, where $\nu_i$ is
the probability of solubility given that the valuations of $l,\dots,e$
satisfy the conditions in row~$i$ of the table, with $\nu_1=\tau_4$.
\[
\begin{array}{rlcccccccccc}
& & & l & m & n & & a & b & c & d & e \\
\tau_4 = & \nu_1 =  \tfrac{1}{2}(1-\tfrac{1}{p}) + \tfrac{1}{p} \nu_2, & &
\ge2 & \ge1 & \ge0 & &
=3 & \ge 3 & \ge 2 & \ge 1 & \ge 0 \\
& \nu_2 = (1-\tfrac{1}{p}) +  \tfrac{1}{p} \nu_3, & &
\ge2 & \ge1 & \ge1 & &
=3 & \ge 3 & \ge 2 & \ge 1 & \ge 1 \\
& \nu_3 = \tfrac{1}{p} \nu_4, & &
\ge2 & \ge1 & \ge1 & &
=3 & \ge 3 & \ge 2 & \ge 2 & \ge 1 \\
& \nu_4 = \tfrac{1}{2}(1-\tfrac{1}{p}) \frac{p}{p+1} +
\tfrac{1}{2}(1-\tfrac{1}{p}) + \tfrac{1}{p} \nu_5,  & &
\ge1 & \ge0 & \ge0 & &
=1 & \ge 1 & \ge 0 & \ge 0 & \ge 0 \\
& \nu_5 = (1-\tfrac{1}{p}) +  \tfrac{1}{p} \nu_6, & &
\ge1 & \ge1 & \ge0 & &
=1 & \ge 1 & \ge 1 & \ge 0 & \ge 0 \\
& \nu_6 =  \tfrac{1}{2}(1-\tfrac{1}{p}) + \tfrac{1}{p} \nu_7, & &
\ge1 & \ge1 & \ge0 & &
=1 & \ge 1 & \ge 1 & \ge 1 & \ge 0 \\
& \nu_7 = \sigma_4'& &
\ge1 & \ge1 & \ge1 & &
=1 & \ge 1 & \ge 1 & \ge 1 & \ge 1\\
\end{array}
\]
Assuming the correctness of this table, we can express $\tau_4$ in
terms of~$\sigma_4'$.  But we also have
\[
\sigma_4' = \sum_{i=0}^{4}\eta_i'\tau_i =
\eta_1' + \eta_2'\frac{1}{2} + \eta_3'\frac{3}{4} + \eta_4'\tau_4,
\]
using the previously established values of $\tau_i$ for $0\le i\le 3$
and the weights~$\eta_i'$ instead of~$\eta_i$, as given in
Corollary~\ref{cor:q_probs}. Solving the two equations gives the value
of $\tau_4$ as stated in the proposition, and also
\[
\sigma_4' = \frac{5p^{10} + 5p^{9} - p^{7} + 3p^{6} - 4p^{5} + 4p^{3} - 8p - 4}
      {8 (p+1) (p^9-1)}.
\]

Now we justify the reduction steps and the relations between
successive $\nu_i$.  At each step, we determine whether there is a
smooth solution which lifts, or there is no solution, or we continue
to the next line; the probability of continuing to the next line is
always~$1/p$.

\begin{enumerate}
\item In line 1, the equation reduces to
  $z^2+nz-e\equiv0\pmod{p}$. With probability $\frac12(1-1/p)$ this
  has no roots over~$\Fp$ and the equation is insoluble; with the
  same probability it has simple $\Fp$-roots which lift, so the
  equation is soluble; and with probability $1/p$ there is a double
  root.  In the latter case we shift the root to $z\equiv0$, leading
  to line~2.
\item With probability $1-1/p$ we have $v(d)=1$; then $\frac{1}{p}f$
  is linear modulo~$p$, so has a simple root which lifts, and we
  obtain a solution (with $z=0$). Otherwise, $v(d)\ge2$, leading to
  line~3.
\item With probability $1-1/p$ we have $v(e)=1$; then there are no
  solutions.  Otherwise rescale, replacing $z$ by~$pz$ and dividing
  through by~$p^2$, leading to line 4.
\item The reduced equation now has the form $z^2 + (mx+n)z \equiv
  cx^2+dx+e$, a possibly singular conic.  If $z^2+mz-c$ is
  irreducible (which has probability~$\frac{1}{2}(1-1/p)$), then by 
  Lemma~\ref{lem:two-lemmas-1} we have solubility with
  probability~$\oldbeta=p/(p+1)$.  If $z^2+mz-c$ splits over~$\Fp$
  (again with probability $\frac{1}{2}(1-1/p)$), then the conic has
  two distinct $\Fp$-points at infinity, so (whether or not it is
  singular) certainly has at least one more smooth $\Fp$-point.
  Lastly, with probability~$1/p$ the quadratic has a double root
  modulo~$p$; we shift it to $z\equiv0$, leading to line~5.
\item With probability $1-1/p$ we have $v(d)=0$; then the reduced
  equation is linear in~$x$ and we have solubility.  Otherwise
  $v(d)\ge1$, leading to line~6.
\item The reduced equation is just as in line~1, and with probability~$1/p$ 
  we reach line 7.
\item In line 7, the probability $\nu_7$ is the probability of
  solubility of any generalized binary quartic that reduces to $z^2$
  modulo~$p$, given also that $v(a)=1$; this, by definition, is~$\sigma_4'$.
\end{enumerate}

\end{proof}

\subsubsection{Conclusion of the proof of Theorem~\ref{thm:genbinquartic}}
\label{subsubsec:conclusion}
We have established two linear equations (\ref{eqn:rho-from-rho4}) and
(\ref{eqn:rho4-from-rho}) relating $\rho$ and $\sigma_4$.  Solving for
these, we obtain
\[   \sigma_4 = \frac{5 p^{10} + 8 p^9 + p^8 - p^7 + 2 p^6 - 3 p^5 + 4 p^3
  - 10 p - 6}{ 8 (p+1) (p^9-1)}, \]
and finally
\begin{align*}
\rho &= \frac{8 p^{10} + 8 p^9 - 4 p^8 + 2 p^6 + p^5 - 2 p^4 + p^3 -
  p^2 - 8 p - 5}{ 8 (p+1) (p^9-1)}\\
     &= 1 - \frac{4p^7 + 4p^6 + 2p^5 + p^4 + 3p^3 + 2p^2 + 3p + 3}
     {8  (p + 1)  (p^2 + p + 1)  (p^6 + p^3 + 1)}
\end{align*}
as stated in Theorem~\ref{thm:genbinquartic}.

\subsection{The density of soluble binary quartics over $\Zp$}
\label{subsec:p=2}

To complete the proof of Theorem~\ref{maingenusonelocal}, we need to
know the density~$\rho(p)$ of soluble binary quartics \eqref{hypereq}
over~$\Zp$, as opposed to the density~$\rho'(p)$ of generalized binary
quartics (\ref{eq:gen_bq}) over~$\Zp$.  For odd primes, it is clear
that we may complete the square without affecting the density, and
obtain the same density as given in Theorem~\ref{thm:genbinquartic},
so $\rho(p)=\rho'(p)=\rho$ (in the notation used above).

Now let $p=2$, and consider the binary quartic
\begin{equation*}
   z^2 = ax^4+bx^3y+cx^2y^2+dxy^3+ey^4.
\end{equation*}
If $b$ or $d$ is odd, then there are smooth points on the reduction
modulo~$2$. If instead $b$ and $d$ are both even, then replacing $z$
by~$z + a x^2 + c xy + e y^2$ gives a generalized binary quartic with
all coefficients even.  The probability of solubility in this case
is~$\sigma_4=4691/6132$, as computed in
\S\ref{subsubsec:conclusion}, giving a final answer of
$\rho(2)=(3/4) + (1/4)\sigma_4 = 23087/24528$.  This completes the
proof of Theorem~\ref{maingenusonelocal}.

\section{The density of soluble binary quartics over $\R$}
\label{sec:R_density}
In this section we use rigorous numerical computational methods to
establish bounds for~$\rho(\infty)$, the probability that
a random binary quartic form~$f$ with real coefficients independently
and uniformly distributed in $[-1,1]$ is not negative definite.
Clearly, $1-\rho(\infty)$ is the probability that $f$ is negative
definite, and the probability that $f$ is positive definite is the
same, so $2-2\rho(\infty)$ is the probability that $f$ has no real
roots.

It suffices to consider inhomogeneous polynomials $f(x)\in\R[x]$.
Writing $f(x) = ax^4+bx^3+cx^2+dx+e$, let $\Delta=\Delta(a,b,c,d,e)$
be the discriminant of~$f$, which is a polynomial in $a,b,c,d,e$ of
degree~$6$, with 16 terms.  We also define the quantities
\begin{align*}
H &= 8ac-3b^2; \\
Q &= 3b^4-16ab^2c+16a^2c^2+16a^2bd-64a^3e.
\end{align*}
Then the condition that $f$ has no real roots is
\begin{equation}\label{posdef-criterion}
\Delta>0, \qquad\text{and}\qquad H>0\quad\text{or}\quad Q<0
\end{equation}
(see \cite[Prop.~7]{JC34}).  Hence $\rho(\infty) = 1-\vol({\mathcal
  R})/64$, where $\vol({\mathcal R})$ is the volume of the region
\[ \mathcal R =
   \{(a,b,c,d,e)\in[-1,1]^5 \mid
   (\Delta>0)\ \text{and}\ ((H>0)\ \text{or}\ (Q<0)) \}.
\]
We have been unable to compute this value exactly by
analytic means.  Instead, we have computed rigorous lower and upper
bounds for~$\vol({\mathcal R})$, and hence for $\rho(\infty)$,
numerically.

\begin{proposition}[$=$ Proposition~\ref{prop:R-density-bounds}]
  \label{prop:R-bounds}
  The probability $\rho(\infty)$ that a random real quartic with
  coefficients independently and uniformly distributed in $[-1,1]$ is
  not negative definite satisfies
\[  0.873954 \le \rho(\infty) \le 0.874124. \]
\end{proposition}

The simplest way to estimate $\rho(\infty)$ non-rigorously is by
Monte Carlo sampling.  Taking $10^7$ sampling points in $[-1,1]^5$ and
using \eqref{posdef-criterion} to test for being positive or negative
definite gives the estimate $\rho(\infty)\approx0.8741239$; using
$10^8$ sampling points gives $\rho(\infty)\approx0.874112095$.  This
suggests that $\rho(\infty)\approx 0.87411$, and one expects this to be
close to the actual value, but we cannot make any rigorous statement
without additional work.

To obtain rigorous bounds as in Proposition
\ref{prop:R-bounds} we have tried several methods, each
implemented as a {\tt C} program for efficiency, using only exact
arithmetic to avoid any rounding errors.  Here we only describe a
basic recursive strategy and sketch one improvement, which takes
advantage of homogeneity to reduce from a
5-dimensional problem to a 4-dimensional one.

The basic recursive method proceeds as follows.  Identify points
$(a,b,c,d,e)\in\R^5$ with quartics~$f=f_{(a,b,c,d,e)}\in\R[x]$.  Given
two vectors $l=(l_0,l_1,l_2,l_3,l_4)$ and $u=(u_0,u_1,u_2,u_3,u_4)$
in~$\R^5$ satisfying $l\le u$ (meaning $l_i\le u_i$ for all $i$), we
consider the 5-dimensional box
\[
B(l,u) = \{f=(a,b,c,d,e)\in\R^5\mid l_0\le a\le u_0, \dots, l_4\le e \le
u_4\} = \{f\in\R^5\mid l\le f\le u\}.
\]
Let $s=(l_0,u_1,l_2,u_3,l_4)$ and $t=(u_0,l_1,u_2,l_3,u_4)$: these are
both corners of the box.  Then for $x\ge0$ we have
\[
f_l(x) \le f(x) \le f_u(x) \qquad\text{for all $f\in B(l,u)$},
\]
while for $x<0$ we have
\[
f_s(x) \le f(x) \le f_t(x) \qquad\text{for all $f\in B(l,u)$}.
\]
It follows that
\begin{itemize}
  \item all $f\in B(l,u)$ are negative definite if and only if both
    $f_u$ and $f_t$ are negative definite;
  \item no $f\in B(l,u)$ are negative definite if either $f_l(x)\ge0$
    for some $x\ge0$, or $f_s(x)\ge0$ for some $x\le0$.
\end{itemize}
Note that the second condition is only sufficient, not necessary.  To
test it, we need to be able to test whether a quartic $f$ takes only
negative values on the positive or negative real half-lines.  In our
code we do this by using a function in the PARI/GP library \cite{PARI2}
based on Descartes' ``rule of signs'', which gives the exact number of
real roots in any interval, using only exact arithmetic for
polynomials with rational coefficients.

Hence, by testing just four quartics, defined by four of the 32
corners of the box $B(l,u)$, we are able to determine whether one
of three cases occurs: (i) all $f\in B(l,u)$ are negative definite;
(ii) no $f\in B(l,u)$ are negative definite, or (iii) neither
(undecided).  In case (iii), we may then divide the box into two
sub-boxes of half the volume by bisecting the longest edge (halving
the maximum value of $u_i-l_i$) and recurse.  We start with the box $B
= [-1,1]^5$ defined by $l=(-1,-1,-1,-1,-1)$ and $u=(1,1,1,1,1)$, and
we also initialise to zero two variables~$v_1$, $v_2$, which will hold
lower bounds for the total volume of sub-boxes containing all,
respectively no, negative definite quartics.  On testing each box, we
either add its volume to one of these variables if case (i) or (ii)
holds, or recurse.  The stopping condition for the recursion is that
we do not recurse when a sub-box is undecided and below a given volume
threshold; equivalently, we bound the depth of recursion.

At the end of this process, we may conclude that the volume we require
is at least~$v_1$ and at most~$32-v_2$, and hence
\[
\frac{1}{32}v_2\le\rho(\infty)\le 1-\frac{1}{32}v_1.
\]
The length of this interval is $1/32$ times the total volume of the
sub-boxes left undecided, which is $32-v_1-v_2$.

Note that all the boxes considered during this process have all their
vertices (and hence their volume) rational, with a denominator which
is a power of~$2$.  Also, all the quartics we test for being positive
or negative have integer coefficients scaled by a power of~$2$, so
this test (using \eqref{posdef-criterion}) is also exact.  Hence we
may use exact arithmetic throughout, so that there are no rounding
errors involved, and obtain exact rational values (with denominator a
power of~$2$) for $v_1$ and~$v_2$, and hence for the bounds
on~$\rho(\infty)$.  Here we express them as decimals to 6 decimal
places for simplicity, rounding down the lower bound and rounding up
the upper bound.

In our implementation, we use various obvious symmetries, such as
reversing the coefficient sequence or changing $x$ to $-x$, to reduce
the running time.  By increasing the depth of recursion, we may reduce
the undecided volume and hence the length of the interval in which
$\rho(\infty)$ certainly lies.  In practice, however, we have found
that the convergence of this process is extremely slow.
In order to speed up the computation and hence obtain tighter bounds,
we implemented the following improvement.  The condition that the
quartic with coefficients $(a,b,c,d,e)$ is negative definite is
obviously homogeneous with respect to scaling by positive constants.
We subdivide the set of quartics according to which coefficient is
greatest in absolute value, and whether it is positive or negative,
giving ten subsets.  (We may ignore the boundary regions where the
maximum is attained at more than one coefficient, since these have
measure zero.)  Some of these subsets are trivial to deal with: for
example, if either the leading coefficient or the constant coefficient
are positive, then the quartic is certainly not negative
definite. Each subset may be scaled so that the maximum coefficient is
$\pm1$, and then a $4$-dimensional recursion similar to the
$5$-dimensional one described above can be used to give lower and
upper bounds on the volume of the negative definite forms in each
subset.  The final step is to add these and scale appropriately
(effectively integrating with respect to the actual maximum
coefficient, from $0$ to $1$) to obtain the lower and upper bounds.

The last scaling step introduces a factor of~$5$ in the denominator,
since $\int_0^1x^4dx=1/5$.  In the table below we round the exact
bounds computed to $6$ decimal places.

Using this scaling method, we were able to increase the depth of
recursion to~$46$.  The following table shows the bounds obtained, and
the computation time (on a single processor),
for recursion depths up to~$46$ (the latter taking nearly 116 days):
\[
\begin{array}{|r|r|r|r|}
\hline
\text{Depth} & \text{Time} & \text{Lower bound} & \text{Upper bound}\\
\hline\hline
20  &       10\text{s}  &      0.863648  &   0.885568 \\
25  &    2\text{m}~24\text{s}  &      0.869623  &   0.878944 \\
30  &   39\text{m}~28\text{s}  &      0.872427  &   0.875876 \\
35  &  516\text{m}~52\text{s}  &      0.873360  &   0.874896 \\
40  & 6620\text{m}~35\text{s}  &      0.873767  &   0.874447 \\
45  & 100988\text{m}~50\text{s} &     0.873930  &   0.874157 \\
46  & 167990\text{m}~53\text{s} &     0.873954  &   0.874124 \\ \hline
\end{array}
\]
This justifies the bounds given in
Proposition~\ref{prop:R-bounds} for $\rho(\infty)\approx
0.874$ (to 3 significant figures).

We also used a Monte Carlo simulation to estimate $\rho'(\infty)$,
the density of generalized binary quartics which are soluble over the
reals.  Taking $10^8$ samples from $[-1,1]^8$ we obtained the value
of~$0.873742745$.  We have not determined rigorous bounds for the
actual value of~$\rho'(\infty)$, but expect it to be a little
smaller than $\rho(\infty)$.

\subsection*{Acknowledgments}
We thank Benedict Gross, Marc Masdeu, Michael Stoll, Terence Tao, and
Xiaoheng Wang for helpful conversations. The first author was
supported by a Simons Investigator Grant and NSF grant~DMS-1001828,
and thanks the Flatiron Institute for its kind hospitality during the
academic year 2019--2020.  The second author was supported by EPSRC
Programme Grant EP/K034383/1 \textit{LMF: L-Functions and Modular
  Forms}, the Horizon 2020 European Research Infrastructures project
\textit{OpenDreamKit} (\#676541), and the Heilbronn Institute for
Mathematical Research.

\end{document}